\documentclass[runningheads,oribibl]{llncs}

\usepackage[utf8]{inputenc}

\usepackage[
	pdftex,
	colorlinks=true,
	pdfpagemode=none,
	urlcolor=blue,
	linkcolor=blue,
	citecolor=blue,
	pdfstartview=FitH
]{hyperref}
\usepackage{url}

\usepackage{ifpdf}

\ifpdf
\usepackage[pdftex]{graphicx}
\else
\usepackage{graphicx}
\fi

\usepackage[rgb]{xcolor}

\usepackage{amsmath,amssymb}
\usepackage{booktabs}

\usepackage{eucal}

\usepackage{relsize}

\usepackage{listings}

\newtheorem{algorithm}{Algorithm} 

\DeclareMathOperator{\real}{Re}

\DeclareMathOperator{\divv}{div}

\newcommand{\e}{\mathrm{e}}
\renewcommand{\i}{\mathrm{i}}
\newcommand{\pair}[1]{\langle #1 \rangle}
\newcommand{\inner}[1]{\langle\!\langle #1 \rangle\!\rangle}

\providecommand{\abs}[1]{\lvert#1\rvert}
\providecommand{\vect}[1]{\boldsymbol{#1}}
\newcommand{\eps}{\varepsilon}
\newcommand{\trans}{\top}

\newcommand{\uud}{\, \mathrm{d}}
\newcommand{\ud}{\mathrm{d}}
\newcommand{\pd}{\partial}

\newcommand{\R}{{\mathbb R}}
\newcommand{\C}{{\mathbb C}}

\newcommand{\D}{\mathbb{D}}

\newcommand{\Diff}{\mathrm{Diff}}
\newcommand{\Emb}{\mathrm{Emb}}
\newcommand{\Xcal}{\mathfrak{X}}

\newcommand{\Fcal}{\mathcal{F}}
\newcommand{\LieD}{\pounds}

\DeclareMathOperator{\ad}{ad}

\providecommand{\SL}{\mathrm{SL}}

\newcommand{\Acal}{\mathcal{A}}

\newcommand{\con}{\mathrm{c}}
\newcommand{\Xcalcon}{{\Xcal_\con}}

\newcommand{\Con}{{\mathrm{Con}}}
\newcommand{\set}[1]{\mathsf{#1}}

\newcommand{\Id}{\mathrm{Id}}

\begin{document}

\pagestyle{headings}
\mainmatter

\title{Geodesic Warps by Conformal Mappings} 

\author{Stephen Marsland\thanks{\url{s.r.marsland@massey.ac.nz}}\inst{1}, Robert I McLachlan\thanks{\url{r.mclachlan@massey.ac.nz}}\inst{2}, Klas Modin\thanks{\url{klas.modin@chalmers.se}}\inst{2}, and 
Matthew Perlmutter\thanks{\url{matthew@mat.ufmg.br}}\inst{2}}
\authorrunning{Marsland, McLachlan, Modin, Perlmutter}
\institute{
School of Engineering and Advanced Technology (SEAT) \\
Massey University, 
Palmerston North, 
New Zealand
\and
Institute of Fundamental Sciences (IFS) \\
Massey University, 
Palmerston North, 
New Zealand
}

\maketitle

\begin{abstract}
	In recent years there has been considerable interest in methods for diffeomorphic warping of images, with applications e.g.\ in medical imaging and evolutionary biology.
	The original work generally cited is that of the evolutionary biologist D'Arcy Wentworth Thompson, who demonstrated warps to deform images of one species into another.
	However, unlike the deformations in modern methods, which are drawn from the full set of diffeomorphism, he deliberately chose lower-dimensional sets of transformations, such as planar conformal mappings. 
	
	In this paper we study warps of such conformal mappings. 
	The approach is to equip the infinite dimensional manifold of conformal embeddings with a Riemannian metric, and then use the corresponding geodesic equation in order to obtain diffeomorphic warps. 
	After deriving the geodesic equation, a numerical discretisation method is developed.
	Several examples of geodesic warps are then given.
	We also show that the equation admits totally geodesic solutions corresponding to scaling and translation, but not to affine transformations. 
\end{abstract}

\section{Introduction}

The use of diffeomorphic transformations in both image registration and shape analysis is now common and utilised in many machine vision and image analysis tasks. 
One image or shape is brought into alignment with another by deforming the image until some similarity measure, such as sum-of-squares distance between pixels in the two images, reaches a minimum. 
The deformation is computed as a geodesic curve with respect to some metric on the diffeomorphism group. 
For a general treatment and an overview of the subject see the monograph by Younes~\cite{Yo2010} and references therein.

The standard approach to the deformation method is to first perform an affine registration
(principally to remove translation and rotation), and then to seek a geodesic warp of the
image in the full set of diffeomorphisms of a fixed domain. Typically the setting 
is to use the right-invariant $H^1_\alpha$ metric,
which leads to the so-called EPDiff equation (see e.g.~\cite{HoMa2005}).
However, in what is arguably the most influential demonstration of the application of
warping methods -- the evolutionary biologist 
D'Arcy Wentworth Thompson's seminal book `On Growth and Form'~\cite{Th1942} --
Thompson warps images of one biological species into another using relatively simple types of transformations, 
so that the gross features of the two images match.
In a recent review of his work, biologist Arthur Wallace says:

\vspace{1ex}

{\em ``This theory cries out for causal explanation, which is something the great man eschewed. 
{\normalfont [\dots\!]}
His transformations suggest coordinated rather than piecemeal changes to development 
in the course of evolution, an issue which almost completely disappeared from view in 
the era of the `modern synthesis' of evolutionary theory, but which is of central 
importance again in the era of evo-devo. {\normalfont [\dots\!]}
All the tools are now in place to examine the mechanistic basis of transformations. 
Not only do we have phylogenetic systematics and evo-devo, but, so obvious that it
is easy to forget, we have computers, and especially, 
in this context, advanced computer graphics.
We owe it to the great man to put these three things together to investigate the 
mechanisms that produce the morphological changes that he captured so elegantly 
with little more than sheets of graph paper and, 
of course, a brilliant mind.''}~\cite{Wa2006b}

\vspace{1ex}

\begin{table}
	\begin{center}
		\begin{tabular}{p{0.57\textwidth}l}
			\toprule 
			\bf Figure no.\ in \cite{Th1942} & \bf Transformation group \\
			\midrule
			515 & $x\mapsto ax,\; y\mapsto y $ \\
			513.2 & $x\mapsto ax,\; y \mapsto b y$ \\ 
			509, 510, 518 & $x\mapsto ax,\; y\mapsto cx + dy$ (shears) \\
			 521--22, 513.5 & $x\mapsto ax+by,\; y\mapsto cx+dy$ (affine) \\
			506, 508 & $x\mapsto ax,\; y\mapsto g(y)$ \\
			511 & $x\mapsto f(x),\; y\mapsto g(y)$ \\
			517--20, 523, 513.1, 513.3, 513.4, 513.6, 514, 525 & conformal \\
			524 & `peculiar' \\
			\bottomrule
		\end{tabular}
	\end{center}
	\caption{Transformation groups used in some transformations in Chapter XII, `On the 
	Theory of Transformations, or the Comparison of Related Forms', of~\cite{Th1942}.}
\end{table}

We draw attention to two key aspects of Thompson's examples: (i) the transformations are as simple as
possible to achieve what he considers a good enough match (see Table 1); and (ii) the classes of
transformations that he considers all forms groups (or pseudogroups), either finite or infinite
dimensional. Mostly, he uses conformal transformations, a constraint he is reluctant to give
up:

\vspace{1ex}

{\em ``It is true that, in a mathematical sense, it is not a perfectly satisfactory or perfectly
regular deformation, for the system is no longer isogonal; but {\normalfont [\dots\!]} 
approaches to an isogonal system under
certain conditions of friction or constraint.''} \cite[p.~1064]{Th1942}

{\em `` {\normalfont [\dots\!]} it will perhaps be
noticed that the correspondence is not always quite accurate in small details. It could easily have been
made much more accurate by giving a slightly sinuous curvature to certain of the coordinates. But as they
stand, the correspondence indicated is very close, and the simplicity of the figures illustrates all the
better the general character of the transformation.''} \makebox{[ibid., p.~1074]}

\vspace{1ex}

For applications in image registration we therefore suggest to vary the group of transformations from which warps are drawn.
If a low dimensional group gives a close match, then it should be preferred over a similar match from a higher-dimensional group. 
If necessary, local deformations from the full diffeomorphism group can be added later to account for fine details. 
In this paper we consider the case of conformal transformations.
More precisely, we consider the problem of formulating and solving a geodesic equation on the space of conformal mappings.
This is a fundamental sub-task in the framework of \emph{large deformation diffeomorphic metric mapping} (LDDMM)~\cite{Tr1995,DuGr1998,Tr1998,JoMi2000,MiYo2001,Be2003,BeMiTrYo2005,BrGaHoRa2011}, which is the standard setup for diffeomorphic image registration.
Based on the geodesic equation derived in this paper, the full conformal image registration problem will be considered in future work.





Although the composition of two conformal maps is conformal, it need not be invertible: we need to restrict
the domain. The invertible conformal maps from the disk to itself do form a group, the disk-preserving
M\"{o}bius group, but it is only 3 dimensional. We are therefore led to consider the infinite dimensional
configuration space $\Con(\set{U},\R^2)$ of conformal embeddings of a simply connected 
compact domain~$\set{U}\subset \R^2$
into the plane. This set is not a group, but it is a pseudogroup.

In~\cite{MaMcMoPe2011} the authors study a geodesic equation
using an $L^2$--metric on the infinite dimensional manifold of conformal embeddings and 
a numerical method is developed for the initial value problem, 
based on the reproducing Bergman kernel. Using numerical examples,
it is shown that the geodesic equation is ill-conditioned as an initial value problem,
and that cusps are developed in finite time which leads to a break-down of the dynamics. 

In this paper we continue the study of geodesics on the manifold of conformal embeddings,
but now with respect to a more general class of Sobolev type $H^1_\alpha$ metrics.
Furthermore, we develop a new numerical algorithm for solving the equations.
The new method is based on a discrete variational principle, and directly solves
the two point boundary value problem.
Our new numerical method
behaves well, i.e., converges fast, for all the examples we tried
as long as the distance between the initial and final point on the manifold
is not too large. From this observation we expect that 
the initial value problem with $\alpha >0$ is well-posed in the
$H^s$ Banach space topology, which would imply that the Riemann exponential
is a local diffeomorphism (see~\cite{La1999,EbMa1970,Sh1998}). This question will
be investigated in detail in future work, as it is out of the scope of the current paper.
The experimental results in this paper (\autoref{sec:results}) are limiting to confirming that the discrete Lagrangian method can reproduce the known geodesics consisting of linear conformal maps and can also calculate non-linear geodesics with moderately large deformations. 
This is a first step towards exploring the metric geometry of the conformal embeddings, similar to what was done for metrics on planar curves by Michor and Mumford \cite{MiMu2006}.

For analysis of 2D shapes, a setting using conformal mappings is developed 
in~\cite{ShMu2006}. There it is shown that the space of planar shapes
is isomorphic to the quotient space $\Diff(S^1)/P\SL(2,\R)$, where
$P\SL(2,\R)$ acts on $\Diff(S^1)$ by right composition of its corresponding disk preserving
Möbius transformation restricted to~$S^1$. Furthermore it is shown
that there is a natural metric on $\Diff(S^1)/P\SL(2,\R)$, the Weil-Peterson metric,
which has non-positive sectional curvature. The setting in this paper is related but
different: rather than studying planar shapes we study conformal transformations between
planar domains and we think of the manifold of conformal embeddings as a submanifold
of the full space of planar embeddings. 
The equation we obtain can be seen as a generalisation and a restriction of the EPDiff equation.
First, a generalisation by going from the group of diffeomorphisms of a fixed domain to the manifold of embeddings from a planar domain into the entire plane.\footnote{
	This generalisation of EPDiff has not yet been worked out in detail in the literature.
	However, it is likely that the approach developed in~\cite{GaMaRa2012} for free boundary flow can be used with only minor modifications. 
}
\footnote{
	One can also look at the generalisation of EPDiff to embeddings from a Klein geometry perspective.
	Indeed, let $\Diff_{\set U}(\R^2)$ denote the diffeomorphisms that leaves the domain $\set U$ invariant.
	Then the embeddings $\Emb(\set U,\R^2)$ can be identified with the space of co-sets $\Diff(\R^2)/\Diff_{\set U}(\R^2)$.
}
Second, a restriction by restricting to conformal embeddings.
The approach we take is similar to (and much influenced by) the recent paper~\cite{GaMaRa2012}, in which a geometric framework for moving boundary continuum equations in physics is developed.



\section{Mathematical Setting and Choice of Metric} 
\label{sec:choice_of_metric}

The linear space of smooth maps $\set U \to \R^2$ is denoted $C^\infty(\set U,\R^2)$.
Recall that this space is a Fréchet space, i.e., it has a topology defined by
a countable set of semi-norms (see~\cite[Sect.~I.1]{Ha1982} for details on the Fréchet topology used).
The full set of embeddings $\set U \to \R^2$, denoted $\Emb(\set U,\R^2)$, is a open
subset of $C^\infty(\set U,\R^2)$. In particular, this implies that~$\Emb(\set U,\R^2)$ is a 
Fréchet manifold (see~\cite[Sect.~I.4.1]{Ha1982}). 

Since $\set U\subset\R^2$ 
it holds that $\Emb(\set U,\R^2)$
contains the identity mapping on~$\set U$, which we denote by~$\Id$.
The tangent space $T_\Id\Emb(\set U,\R^2)$ at the identity
is given by the smooth vector fields on~$\set U$, which we denote by~$\Xcal(\set U)$.
Notice that the vector fields need not be tangential to the boundary $\pd U$.
Also notice that $\Xcal(\set U)$ is a Fréchet Lie algebra
with bracket given by minus the Lie derivative on vector fields,
i.e., if $\xi,\eta\in\Xcal(\set U)$, then $\ad_\xi(\eta) = -\LieD_\xi\eta$.

Let $\mathsf{g} = \ud x\otimes \ud x + \ud y \otimes \ud y$ be the standard
Euclidean metric on~$\R^2$. 
Consider the subspace of $C^\infty(\set U,\R^2)$ 
consisting of maps that preserve the metric up to multiplication 
with elements in the space $\Fcal(\set U)$ of 
smooth real valued functions on~$\set U$. That is, the subspace
\[
C^\infty_\con(\set U,\R^2) = 
\{ \phi\in C^\infty(\set U,\R^2); \phi^*\mathsf{g}= F \mathsf{g}, F\in\Fcal(\set U) \} .
\]
This subspace 
is topologically closed in $C^\infty(\set U,\R^2)$. The set of conformal
embeddings $\Con(\set U,\R^2) = \Emb(\set U,\R^2)\cap C^\infty_\con(\set U,\R^2)$
is an open subset of $C^\infty_\con(\set U,\R^2)$ and 
a Fréchet submanifold of $\Emb(\set U,\R^2)$.
The tangent space $T_\Id \Con(\set U,\R^2)$ is given by
\begin{equation*}
	\Xcalcon(\set U) = \{ \xi \in \Xcal(\set U); \LieD_\xi\mathsf{g} = \divv(\xi)\mathsf{g} \} ,
\end{equation*}
which follows by straightforward calculations. Notice that
$\Xcalcon(\set U)$ is a subalgebra of $\Xcal(\set U)$,
since 
\begin{equation*}
	\begin{split}
		\LieD_{\LieD_\xi\eta} \mathsf{g} &= \LieD_\xi\LieD_\eta\mathsf{g} - \LieD_\eta\LieD_\xi\mathsf{g}
		= \LieD_\xi (\divv(\eta)\mathsf{g}) - \LieD_\eta(\divv(\xi)\mathsf{g})
		\\
		&= \big(\LieD_\xi \divv(\eta) - \LieD_\eta \divv(\xi)\big)\mathsf{g}
		+ \underbrace{\divv(\eta)\LieD_\xi\mathsf{g} - \divv(\xi)\LieD_\eta\mathsf{g}}_{0}
		\\
		&= \divv(\LieD_\xi\eta)\mathsf{g} .
	\end{split}
\end{equation*}

In the forthcoming, we identify the plane $\R^2$ with the complex numbers $\C$
through $(x,y)\mapsto z = x + \i y$.
Hence, the vector fields $\Xcal(\set U)$ are identified with smooth complex valued functions
on~$\set U$, and $\Xcalcon(\set U)$ with the space of holomorphic functions. 

The complex $L^2$ inner product on $\Xcal(\set U)$ is given by
\begin{equation*}
	\inner{\xi,\eta}_{L^2(\set U)} := \int_{\set U} \xi(z)\overline{\eta(z)}\uud A(z) ,
\end{equation*}
where $\ud A = \ud x\wedge \ud y$ is the canonical volume form on~$\R^2$.
Correspondingly, we also have the real $L^2$ inner product given by
\begin{equation*}
	\pair{\xi,\eta}_{L^2(\set U)} := \int_{\set U} \mathsf{g}(\xi,\eta)\ud A = 
	\real \inner{\xi,\eta}_{L^2(\set U)} .
\end{equation*}
Also, we have the more general class of real and complex $H^1_\alpha$ inner products
given by
\begin{equation*}
	\begin{split}
		\pair{\xi,\eta}_{H^1_\alpha(\set U)} & := \pair{\xi,\eta}_{L^2(\set U)} + 
		\frac{\alpha}{2} \pair{\xi_x,\eta_x}_{L^2(\set U)}
		+ \frac{\alpha}{2}\pair{\xi_y,\eta_y}_{L^2(\set U)} , \\
		\inner{\xi,\eta}_{H^1_\alpha(\set U)} & := \inner{\xi,\eta}_{L^2(\set U)} + 
		\frac{\alpha}{2} \inner{\xi_x,\eta_x}_{L^2(\set U)}
		+ \frac{\alpha}{2} \inner{\xi_y,\eta_y}_{L^2(\set U)} ,
	\end{split}
\end{equation*}
where $\alpha \geq 0$ and $\xi_x,\xi_y$ respectively denotes derivatives
with respect to the Cartesian coordinates~$(x,y)$.
Notice that if $\xi,\eta\in\Xcalcon(\set U)$, then
\begin{equation*}
	\inner{\xi,\eta}_{H^1_\alpha(\set U)} = \inner{\xi,\eta}_{L^2(\set U)}
	+ \alpha \inner{\xi',\eta'}_{L^2(\set U)}
\end{equation*}
where $\xi'$ and $\eta'$ denote complex derivatives.


The class of inner products $\inner{\cdot,\cdot}_{H^1_\alpha(\set U)}$ on
$T_\Id\Emb(\set U,\R^2) = \Xcal(\set U)$ induces
a corresponding class of Riemannian
metrics on~$\Emb(\set U,\R^2)$ by
\begin{equation}\label{eq:Halpha_metric}
	T_\varphi\Emb(\set U,\R^2)\times T_\varphi\Emb(\set U,\R^2) \ni
	(U,V) \mapsto \pair{U\circ\varphi^{-1},V\circ\varphi^{-1}}_{H^1_\alpha(\varphi(\set U))} .
\end{equation}
Note that $\varphi^{-1}$ is well-defined as a map $\varphi(\set U)\to\set U$
since $\varphi$ is an embedding.
Also note that the restriction of the metric~\eqref{eq:Halpha_metric} to
the submanifold of diffeomorphisms $\Diff(\set U)\subset \Emb(\set U,\R^2)$ 
yields to the ``ordinary'' $H^1_\alpha$ metric
on~$\Diff(\set U)$ corresponding to the $\mathrm{EPDiff}$~equations.



\section{Derivation of the Geodesic Equation}

In this section we derive the geodesic equation on $\Con(\set U,\R^2)$
for the class of $H^1_\alpha(\set U)$~metrics given by~\eqref{eq:Halpha_metric}.
These equation are given by the Euler-Lagrange equations with respect to
the quadratic Lagrangian on $\Con(\set U,\R^2)$ given by
\begin{equation}\label{eq:HalphaLagrangian}
	L(\varphi,\dot\varphi) = 
	\frac{1}{2}\pair{\dot\varphi\circ\varphi^{-1},\dot\varphi\circ\varphi^{-1}}_{H^1_\alpha(\varphi(\set U))}
	= \frac{1}{2}\inner{\dot\varphi\circ\varphi^{-1},\dot\varphi\circ\varphi^{-1}}_{H^1_\alpha(\varphi(\set U))} ,
\end{equation}
where $\dot\varphi\in T_\varphi\Emb(\set U,\R^2)$ corresponds to the time derivative.

As a first step, we have the following result.

\begin{lemma}\label{lem:metric_on_U}
	For any $(\varphi,\dot\varphi)\in T\Emb(\set U,\R^2)$ it holds that
	\begin{equation*}
		\inner{\dot\varphi\circ\varphi^{-1},\dot\varphi\circ\varphi^{-1}}_{H^1_\alpha(\varphi(\set U))}
		= \inner{\varphi'\dot\varphi,\varphi'\dot\varphi}_{L^2(\set U)}
		+ \alpha \inner{\dot\varphi',\dot\varphi'}_{L^2(\set U)} .
	\end{equation*}
\end{lemma}

\begin{proof}
	Let $f$ be any complex valued function on $\set U$.
	By a change of variables $w = \varphi(z)$ we obtain
	\begin{equation*}
		\int_{\varphi(\set U)} f\circ\varphi^{-1}(w) \ud A(w) = \int_{\set U} f(z) \abs{\varphi'(z)}^2\ud A(z).
	\end{equation*}
	For the first term we take $f(z) = \dot\varphi(z)\overline{\dot\varphi(z)}$
	which yields
	\begin{equation*}
		\inner{\dot\varphi\circ\varphi^{-1},\dot\varphi\circ\varphi^{-1}}_{L^2(\varphi(\set U))}
		= \int_{\set U} \dot\varphi(z)\overline{\dot\varphi(z)} \abs{\varphi'(z)}^2\ud A(z)
		= \inner{\varphi'\dot\varphi,\varphi'\dot\varphi}_{L^2(\set U)}.
	\end{equation*}
	For the second term we take first notice that
	\begin{equation*}
		(\dot\varphi\circ\varphi^{-1})'(w) = \frac{\dot\varphi'\circ\varphi^{-1}(w)}{\varphi'\circ\varphi^{-1}(w)}
	\end{equation*}
	and then we take $f(z) = \dot\varphi'(z)/\varphi'(z)$. The result now follows.
	\qed
\end{proof}

We are now ready to derive the Euler-Lagrange equation from the variational principle.
Indeed, we look for at curve $\varphi:[0,1]\to \Con(\set U,\R^2)$ such that
\begin{equation*}
	\frac{\ud}{\ud\eps}\Big|_{\eps=0}\int_0^1
	\frac{1}{2}\inner{\dot\varphi_\eps(t)\circ\varphi_\eps(t)^{-1},
	\dot\varphi_\eps(t)\circ\varphi_\eps(t)^{-1}}_{H^1_\alpha(\varphi(\set U))}
	\uud t = 0
\end{equation*}
for all variations $\varphi_\eps:[0,1]\to \Con(\set U,\R^2)$ such
that $\varphi_\eps(0)= \varphi(0)$, $\varphi_\eps(1) = \varphi(1)$ and
$\varphi_0 = \varphi$. To simplify notation we introduce
$$
\psi = \frac{\ud}{\ud \eps}\Big|_{\eps=0}\varphi_\eps .
$$
Notice that $\psi(0) = \psi(1) = 0$.
Using \autoref{lem:metric_on_U} and the fact that
differentiation commutes with integration we obtain
\begin{equation*}
	\begin{split}
		0 &= \int_0^1 \Big( 
			\inner{\varphi'\dot\varphi,\frac{\ud}{\ud\eps}\Big|_\eps \varphi_\eps' \dot\varphi_\eps}_{L^2(\set U)}
			+ \alpha \inner{\dot\varphi',\frac{\ud}{\ud\eps}\Big|_\eps \dot\varphi_\eps'}_{L^2(\set U)}
		\Big) \ud t \\
		&= \int_0^1 \Big(
			\inner{\varphi' \dot\varphi,\psi'\dot\varphi+\varphi'\dot\psi}_{L^2(\set U)}
			+ \alpha \inner{\dot\varphi',\dot\psi'}_{L^2(\set U)}
		\Big)\ud t \\
		&= \int_0^1 \Big(
			\inner{\varphi' \dot\varphi,\psi'\dot\varphi-\dot\varphi'\psi + \frac{\ud}{\ud t} (\varphi'\psi)}_{L^2(\set U)}
			- \alpha \inner{\ddot\varphi',\psi'}_{L^2(\set U)}
		\Big)\ud t  \\
		&= \int_0^1 \Big(
			-\inner{\frac{\ud}{\ud t}(\varphi' \dot\varphi),\varphi'\psi}_{L^2(\set U)}
			+ \inner{\varphi' \dot\varphi,
			\psi'\dot\varphi-\dot\varphi'\psi}_{L^2(\set U)}
			- \alpha \inner{\ddot\varphi',\psi'}_{L^2(\set U)}
		\Big)\ud t ,
	\end{split}
\end{equation*}
where in the last two equalities we use integration by parts over the time variable 
and the fact that $\psi$ vanishes at the endpoints. Notice that
there are now no time derivatives on~$\psi$. Thus, by the fundamental lemma of 
calculus of variations we can remove the time integration and thereby obtain
a weak equation which must be fulfilled at each point in time. 
In order to
obtain a strong formulation, we need also to isolate $\psi$ from spatial derivatives.
The standard approach of using integration by parts introduces a boundary integral term.
In most examples of calculus of variations, this boundary term
either vanishes (in the case of a space of tangential vector fields),
or it can be treated separately giving rise to natural boundary conditions
(in the case of a space 
where vector fields can have arbitrary small compact support). 
However, in the case
of conformal mappings, there is always a global dependence between
interior points, and points on the boundary (since holomorphic functions cannot have local support).
Hence, we need an appropriate analogue of
integration by parts which avoids boundary integrals.
For this, consider the adjoint operator of complex
differentiation, i.e., an
operator $\pd^{\trans}_{z}:\Xcalcon(\set U)\to\Xcalcon(\set U)$ such that
\begin{equation*}
	\inner{\xi,\eta'}_{L^2(\set U)} = \inner{\pd^{\trans}_{z}\xi,\eta}_{L^2(\set U)}, 
	\qquad \forall\; \xi,\eta\in\Xcalcon(\set U).
\end{equation*}
Notice that $\pd^\trans_z$ depends on the domain~$\set U$.
In the case of the unit disk $\set U = \D$, it holds
that $\pd^\trans_z\xi(z) = \pd_z (z^2 \xi(z)) = 2 z\xi(z) + z^2\xi'(z)$.
In the general case, this operator is more complicated, but can
still be computed under the assumption that a conformal embedding
$\set U \to \D$ is known (see~\cite{MaMcMoPe2011_hodgepreprint}).

Using the operator $\pd^\trans_z$ we can now proceed as follows
\begin{equation*}
	\begin{split}
		0 &= -\inner{\frac{\ud}{\ud t}(\varphi' \dot\varphi),\varphi'\psi}_{L^2(\set U)}
		+ \inner{\varphi' \dot\varphi,
		\psi'\dot\varphi-\dot\varphi'\psi}_{L^2(\set U)}
		- \alpha \inner{\ddot\varphi',\psi'}_{L^2(\set U)}
		\\
		&= 
		-\inner{\abs{\varphi'}^2 \ddot\varphi + \dot\varphi\dot\varphi'\overline{\varphi'},\psi}_{L^2(\set U)}
		-\inner{\dot\varphi\varphi'\overline{\dot\varphi'},\psi}_{L^2(\set U)}
		\\
		&\qquad +\inner{\varphi' \dot\varphi,
		(\psi\dot\varphi)' - \psi\dot\varphi'}_{L^2(\set U)}
		- \alpha \inner{\pd_z^\trans\ddot\varphi',\psi}_{L^2(\set U)}
		\\
		&= 
		- \inner{\abs{\varphi'}^2 \ddot\varphi + \dot\varphi\dot\varphi'\overline{\varphi'}
		+ \dot\varphi\varphi'\overline{\dot\varphi'} + \alpha \pd_z^\trans\ddot\varphi',\psi}_{L^2(\set U)}
		+ \inner{\pd_z^\trans (\varphi' \dot\varphi),
		\psi\dot\varphi}_{L^2(\set U)}
		\\
		&=
		- \inner{(\abs{\varphi'}^2 + \alpha\pd_z^\trans\pd_z) \ddot\varphi + \dot\varphi(\dot\varphi'\overline{\varphi'}
		+ \varphi'\overline{\dot\varphi'}) - 
		\overline{\dot\varphi}\pd_z^\trans (\varphi' \dot\varphi),\psi}_{L^2(\set U)}.
	\end{split}
\end{equation*}
Thus, this relation must hold for all holomorphic functions~$\psi$. However,
the expression in the first slot of the inner product is not holomorphic,
so it needs to be orthogonally projected back to the set of holomorphic functions.
Using Hodge theory for manifolds with boundary, one can show that the orthogonal
complement of $\Xcalcon(\set U)$ in $\Xcal(\set U)$ with respect to
the real inner product $\pair{\cdot,\cdot}_{L^2(\set U)}$ is given by
\begin{equation*}
	\Xcalcon(\set U)^\bot = \{ \xi \in \Xcal(\set U) ; \xi(z) = \pd_x F - \i \pd_y F + \pd_y G + \i \pd_x G, \;
	F,G\in\Fcal_0(\set U) \} ,
\end{equation*}
where $\Fcal_0(\set U) = \{ F\in\Fcal(\set U); F|_{\pd \set U} = 0 \}$ are
the smooth functions that vanish at the boundary.
This result is obtained in~\cite{MaMcMoPe2011_hodgepreprint}.
Since $\Xcalcon(\set U)^\bot$ is invariant under multiplication with~$\i$,
it is also the orthogonal complement with respect to the complex~$L^2$ inner product.

Now we finally arrive at the strong formulation of the Euler-Lagrange equations
\begin{equation}\label{eq:Euler_Lagrange_strong}
	\begin{split}
		\frac{\ud}{\ud t} \big(\Acal(\varphi) \dot\varphi\big)
		%
		-\overline{\dot\varphi}\pd_z^\trans (\varphi' \dot\varphi)
		& = \Acal(\varphi)\big(
			\pd_x F - \i \pd_y F + \pd_y G + \i \pd_x G
		\big),
		\\
		\pd_{\bar z}\varphi &= 0, \\
		F|_{\pd \set U } = G|_{\pd \set U} &= 0 ,
		%
	\end{split}
\end{equation}
where $\Acal(\varphi) = \abs{\varphi'}^2 + \alpha\pd_z^\trans\pd_z$
is the \emph{inertia operator} (self adjoint with respect to the $L^2$ inner product) and
where the second to last equation means that $\varphi$ is
constrained to be holomorphic. Indeed, one may think of equation~\eqref{eq:Euler_Lagrange_strong}
as a Lagrange-D'Alembert equation for a system with configuration space
$\Emb(\set U,\R^2)$ which, by Lagrangian multipliers $(F,G)$,
is constrained to the submanifold $\Con(\set U,\R^2)$.

In the special case $\set U = \D$ we get
\begin{equation}\label{eq:Euler_Lagrange_strong_disk}
	\begin{split}
		\frac{\ud}{\ud t} \big(\Acal(\varphi) \dot\varphi\big)
		-\overline{\dot\varphi}\pd_z (z^2 \varphi' \dot\varphi)
		& = \Acal(\varphi)\big(
			\pd_x F - \i \pd_y F + \pd_y G + \i \pd_x G
		\big),
		\\
		\pd_{\bar z}\varphi &= 0,
		\\
		F|_{\pd \set U } = G|_{\pd \set U} &= 0 .
	\end{split}
\end{equation}

\subsection{Weak Geodesic Equation in the Right Reduced Variable} 
\label{sub:weak_form_of_the_equation_in_right_reduced_variables}

It is also possible to derive the geodesic equation using the right
reduced variable $\xi = \dot\varphi\circ\varphi^{-1}$, as is typically done
for geodesic equations on diffeomorphism groups with invariant metric
(see e.g.~\cite{ArKh1998,KhWe2009,MoPeMaMc2011}). However, there is a difference
between the setting of embeddings and that of diffeomorphism groups, since
$\xi$ is defined on $\varphi(\set U)$, which is not fixed in the embedding
setting. Nevertheless, the ``moving domain'' $\Sigma = \varphi(\set U)$
simply moves along the flow, i.e., points on the boundary follows the flow
of the vector field~$\xi$. For details of this setting in the two cases
of unconstrained embeddings and volume preserving embeddings,
see~\cite{GaMaRa2012}. 

Let $\gamma_\eps$ be a variation of $\gamma$ as above, 
and let $\xi_\eps = \dot\gamma_\eps\circ\gamma_\eps^{-1}$. Using the calculus of Lie derivatives,
direct calculations yield
\[
	\begin{split}
		\frac{\ud}{\ud \eps}\Big|_{\eps=0}\xi_{\eps}
		 &= \dot\eta + \LieD_\eta\xi
		\\
		\frac{\ud}{\ud \eps}\Big|_{\eps=0} \frac{1}{2}\pair{\xi,\xi}_{H^1_\alpha(\gamma_\eps(\set U))}
		&= \pair{\xi,\LieD_\eta\xi}_{L^2(\gamma(\set U))} + 
		\alpha \pair{\xi',\LieD_\eta\xi'}_{L^2(\gamma(\set U))} +
		\\
		& \quad\; \pair{\xi,\divv(\eta)\xi}_{L^2(\gamma(\set U))}
		+ \alpha \pair{\xi',\divv(\eta)\xi'}_{L^2(\gamma(\set U))}
	\end{split}
\]
where $\eta = \frac{\ud}{\ud \eps}\big|_{\eps=0}\varphi_\eps\circ\varphi^{-1}$.
From the variational principle
\begin{equation*}
	\frac{\ud}{\ud\eps}\Big|_{\eps=0}\int_0^1 L(\varphi_\eps,\dot\varphi_\eps) \uud t
	=
	\frac{\ud}{\ud\eps}\Big|_{\eps=0}
	\int_0^1 \frac{1}{2}\pair{\xi_\eps,\xi_\eps}_{H^1_\alpha(\gamma_\eps(\set U))} \uud t = 0
\end{equation*}
we now obtain the weak form of the equation in terms of the variables
$(\varphi,\xi)$ as
\begin{equation*}
	\begin{split}
		\int_0^1  \Big( &
			\pair{\xi,\dot\eta + 2 \LieD_\eta\xi + \divv(\eta)\xi}_{L^2(\gamma(\set U))} + \\
			 & \alpha \pair{\xi',\dot\eta' + \pd_z \LieD_\eta\xi + \LieD_\eta\xi' + \divv(\eta)\xi'}_{L^2(\gamma(\set U))}
		\Big) \uud t = 0 .
	\end{split}
\end{equation*}
Passing now to the complex inner product, we use the formulas
\begin{equation*}
	\begin{split}
		\inner{\xi,\eta}_{L^2(\set U)} &= \pair{\xi,\eta}_{L^2(\set U)} + \i \pair{\xi,\i \eta}_{L^2(\set U)}
		\\
		\mathsf{g}(\xi,\divv(\eta)\xi)+ \i \mathsf{g}(\xi,\divv(\i\eta)\xi) &= 2\xi\overline{\eta'\xi}
		\\
		\LieD_{\i \eta}\xi &= \i \LieD_{\eta}\xi
	\end{split}
\end{equation*}
which yields the weak equation
\begin{multline}\label{eq:weak_eq_reduced_variables_complex}
	\int_0^1 \Big(
		\inner{\xi,\dot\eta + 2\LieD_\eta\xi + 2\eta'\xi}_{L^2(\gamma(\set U))}\\
		+ \alpha\inner{\xi',\dot\eta' + \pd_z(\eta'\xi) + 2\LieD_\eta\xi'}_{L^2(\gamma(\set U))}
	\Big)\uud t = 0 .
\end{multline}
Together with the equation $\dot \varphi = \xi\circ\varphi$, this is thus a weak form of the
equation~\eqref{eq:Euler_Lagrange_strong}, but expressed in the variables $(\varphi,\xi)$.



\section{Totally Geodesic Submanifolds}\label{sec:tg_submanifold}

In this section we investigate special solutions to
equation~\eqref{eq:Euler_Lagrange_strong_disk}. The approach for doing
so is to find a finite dimensional submanifold of $\Con(\D,\R^2)$
such that solutions curves starting and ending on this
submanifold actually lie on the submanifold.

Recall that a submanifold $N\subset M$ of a Riemannian manifold $(M,\mathsf{g})$ is
\emph{totally geodesic with respect to} $(M,\mathsf{g})$ if geodesics 
in $N$ (with respect to $\mathsf{g}$ 
restricted to $N$) are also geodesics in $M$. For a thorough treatment of
totally geodesic subgroups of $\Diff(M)$ with respect to various metrics,
see~\cite{MoPeMaMc2011}.

Consider now the submanifold of linear conformal transformations
\[
\mathrm{Lin}(\D,\R^2) = 
\left\{ \varphi\in\Con(\D,\R^2); 
\varphi(z) = 
c z, c \in \C
\right\}.
\]


\begin{proposition}\label{pro:tg_submanifold1}
	$\mathrm{Lin}(\D,\C)$ is totally geodesic in $\Con(\D,\C)$ with respect to
	the $H^1_\alpha$ metric given by~\eqref{eq:Halpha_metric}.
\end{proposition}

\begin{proof}
	If $t\mapsto \varphi(t)$ is a path in $\mathrm{Lin}(\D,\C)$,
	i.e., $\varphi(z) = c z$ with $c \in\C$,
	then $\xi=\dot\varphi\circ\varphi^{-1}$ is of the form $\xi(z) = a z$
	with $a\in\C$. Now, let $t\mapsto (\varphi,\xi)$
	fulfill the variational equation~\eqref{eq:weak_eq_reduced_variables_complex} 
	for each variation
	of the form $\eta(z) = b z$ with $b\in\C$.
	We need to show that $t\mapsto (\varphi,\xi)$ then fulfills the
	equation for any variation of the form $\eta(z) = z^k$ (since the monomials span
	the space of holomorphic functions).
	Thus, for the first term in~\eqref{eq:weak_eq_reduced_variables_complex} we get
	\begin{equation}\label{eq:tg_weak_form}
		\begin{split}
			\inner{\xi,\dot\eta + 2\LieD_\eta\xi + 2\eta'\xi}_{L^2(\varphi(\D))}
			&=
			\inner{\xi,\dot\eta + 4\eta'\xi-2\xi'\eta}_{L^2(\varphi(\D))}
			\\
			&= \inner{\varphi'\cdot\xi\circ\varphi),\varphi'\cdot(\dot\eta +
			 4\eta'\xi-2\xi'\eta)\circ\varphi}_{L^2(\D)}
			\\
			&= \abs{c}^4\inner{a z,\dot b z^k + 4 k b a z^k - 2 a b z^{k}}_{L^2(\D)} .
		\end{split}
	\end{equation}
	where in the first line we use the conformal change of variables formula
	for integrals. Now, since the monomials are orthogonal with respect to
	$\inner{\cdot,\cdot}_\D$ the expression vanish whenever~$k\neq 1$.
	For the second term, notice that
	$\xi' = a$ is constant. Now, if $\eta$
	is constant, then all the terms $\dot\eta',\eta'\xi,\LieD_\eta\xi'$
	vanish. If $\eta=z^{k}$ with $k\geq 2$, then
	the second term $\alpha\inner{\xi',\dot\eta'+ \pd_z(\xi\eta') + 2\LieD_\eta\xi'}$ 
	is of the form
	\begin{equation*}
		\alpha \int_{c\D} a\overline{(\dot b + 3 ab) k z^{k-1}} \uud A(z)
		= \alpha \abs{c}^2 a \overline{(\dot b + 3 ab)} \int_\D k z^{k-1}
	\end{equation*}
	which vanishes for every $a,c_1\in\C$.
	This concludes the proof.
	\qed
\end{proof}

We now derive a differential equation
for the totally geodesic solutions in $\mathrm{Lin}(\D,\C)$
in term of the variables $(c,a)$ corresponding to 
$\varphi(z)= c z$ and $\xi(z) = a z$. From $\dot\varphi = \xi\circ\varphi$
it follows that $\dot c = a c$. Next, we plug the ansatz
into the weak equation~\eqref{eq:weak_eq_reduced_variables_complex},
and use that $b$ vanish at the endpoints,
which yields the equations
\begin{equation}\label{eq:tg_strong_eq}
	\begin{split}
		\dot c &= a c \\
		\dot a (2 c + \alpha) &= -4a^2 c - \alpha a^2
	\end{split}
\end{equation}
These equations thus gives special solutions to equation~\eqref{eq:Euler_Lagrange_strong_disk}.
We obtain that $\frac{\ud^2}{\ud t^2}(c^2 + \alpha c) = 0$, so we can analytically
compute these special solutions.
%
%
%
%
Notice that if
$a$ and $c$ are initially real, then both $a$ and $c$ stay real, 
so the even smaller submanifold of pure scalings
is also totally geodesic.


\autoref{fig:tg_solutions} gives a visualization of total geodesic solutions 
where the two end transformations are given first by a pure scaling and
second by a pure rotation. Notice that within the submanifold
$\mathrm{Lin}(\D,\C)$, the smaller submanifold of scalings is totally geodesic,
as is shown in the left figure.
However, the submanifold of rotations is not, as is shown in the right figure.


\begin{figure}
	\centering
		\includegraphics[height=0.17\textwidth]{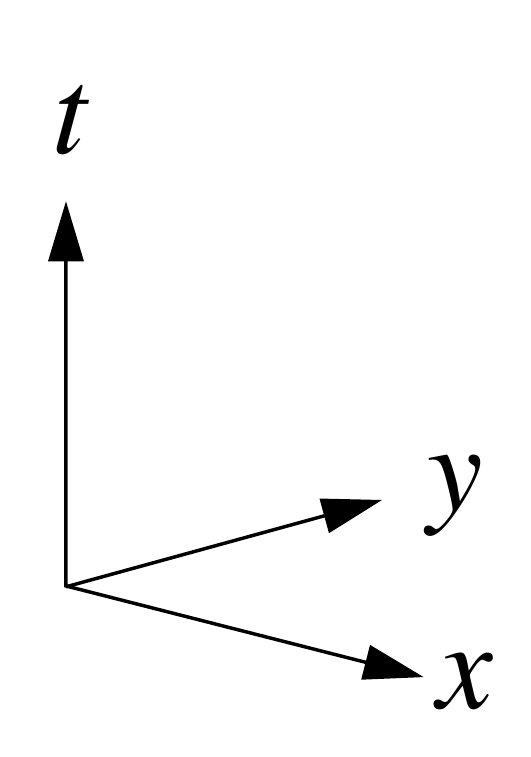}
		\quad
		\includegraphics[height=0.4\textwidth]{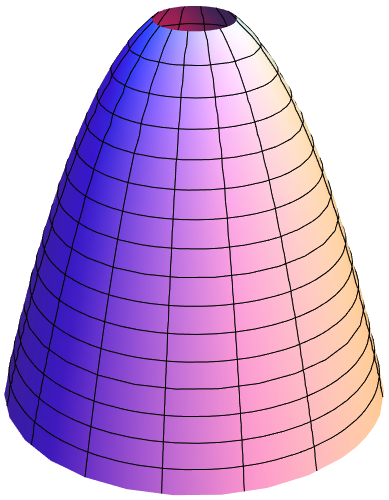}
		\qquad
		\includegraphics[height=0.4\textwidth]{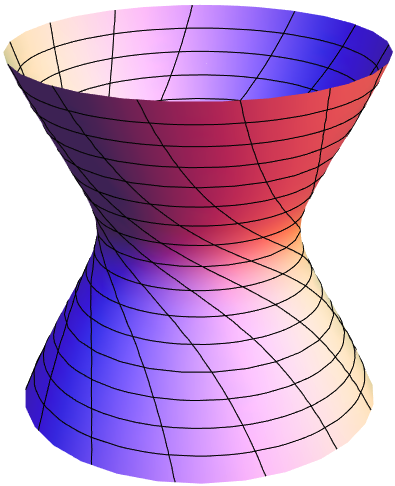}
		\\[2ex]
		\hspace{0.18\textwidth} $c = 0.2$ \hspace{0.27\textwidth} $c=\e^{1.6 \pi \i}$
	\caption{Geodesic curve from $\varphi_0(z)=z$ to $\varphi_1(z)=c z$ 
	for different values of $c$ and $\alpha=0$. 
	The mesh lines show how the unit circle evolves. Notice that the scaling geodesic
	stays a scaling (left figure), whereas the rotation geodesic picks up some scaling
	during its time evolution (right figure).
	}
	\label{fig:tg_solutions}
\end{figure}

\begin{remark}
	By using again the weak form~\eqref{eq:weak_eq_reduced_variables_complex} of the governing
	equation one can further show that the submanifold of translations
	is \emph{not} totally geodesic in $\Con(\D,\C)$.
	Nor is the submanifold 
	of affine conformal
	transformations.
\end{remark}

\section{Numerical Discretization} 
\label{sec:numerical_discretization}

In this section we describe a method for numerical discretization of
the equations~\eqref{eq:Euler_Lagrange_strong_disk}. The basic idea is to obtain
a spatial discretization of the phase space variables $(\varphi,\dot\varphi)\in T\Con(\D,\R^{2})$
by truncation of the Taylor series. Thus, we use a Galerkin type
approach for spatial discretization. For time discretization we take
a variational approach, using the framework of discrete mechanics
(cf.~\cite{MaWe2001}).

The discrete configuration space is given by
\begin{equation*}
	Q_{n} = \{ \varphi \in \Con(\D,\R^{2}) ; \varphi(z) = \sum_{i=0}^{n-1} 
	c_{i} z^{i}, c_{i}\in\C \} .
\end{equation*}
Notice that $Q_{n}$ is an $n$--dimensional submanifold of $\Con(\D,\R^{2})$.
Since $\Con(\D,\R^{2})$ is a open subset of the vector space of
all holomorphic maps on $\D$ (in the Fréchet topology), it holds that
the discrete configuration space~$Q_{n}$
is an open subset of~$\mathrm{span}_\C\{ 1,z,\ldots,z^{n-1} \} \simeq \C^{n}$. 
Thus, each tangent space $T_{\varphi}Q_{n}$
is identified with~$\C^{n}$ by taking the coefficients 
of the finite Taylor series.
Together with the restricted~$H_{\alpha}^{1}$ metric,
$Q_{n}$ is a Riemannian manifold.

Let $U,V\in T_{\varphi}\Con(\D,\R^2)$, and let
$(a_k)_{k=0}^\infty,(b_k)_{k=0}^\infty$ 
respectively be their Taylor coefficients.
Then it holds that
\begin{equation*}
	\inner{U,V}_{L^2(\D)}
	= \frac{1}{\pi}\sum_{k=0}^{\infty}
		(1 + i) a_{i}\overline{b_{i}} \; ,
\end{equation*}
which follows since $\inner{z^{i},z^{j}}_{L^{2}(\D)} = \delta_{ij} (i+1)/\pi$.

The next step is to obtain
a numerical method that approximates the geodesics.
For this we use the variational method obtained by the
discrete Lagrangian on $Q_{n}\times Q_{n}$ given by
\begin{equation*}
	\begin{split}
		L_{d}(\varphi_{k},\varphi_{k+1}) &= h L \big( 
			\underbrace{\frac{\varphi_{k}+\varphi_{k+1}}{2}}_{\varphi_{k+1/2}},
			\frac{\varphi_{k+1}-\varphi_{k}}{h}
			\big)
		 \\ &=
		\frac{1}{2 h}\inner{
			\varphi_{k+1/2}'(\varphi_{k+1}-\varphi_{k}),
			\varphi_{k+1/2}'(\varphi_{k+1}-\varphi_{k})
		}_{L^2(\D)} +
		\\
		& \quad\;
		\frac{\alpha}{2 h}\inner{
			\varphi_{k+1}'-\varphi_{k}',
			\varphi_{k+1}'-\varphi_{k}'
		}_{L^2(\D)} ,
	\end{split}
\end{equation*}
where $h>0$ is the step size.
The discrete action is thus given by
\begin{equation}\label{eq:discrete_action}
	A_{d}(\varphi_{0},\ldots,\varphi_{N}) = \sum_{k=0}^{N-1}L_{d}(\varphi_{k},\varphi_{k+1}) .
\end{equation}
Now, a method for numerical computation of geodesics originating from the identity and ending at a known configuration is obtained as follows.

\begin{algorithm}
	Given $\varphi\in \Con(\D,\R^{2})$, an approximation to the geodesic curve from the identity element $\varphi_{0}(z) = z$ to $\varphi$ is given by the following algorithm:
	\begin{enumerate}
		\item Set $\varphi_N = \mathcal{P}_n\varphi$, where $\mathcal{P}_n:\Con(\D,\R^2)\to Q_n$
		is projection by truncation of Taylor series.
		\item Set initial guess $\varphi_k = (1-k/N)\varphi_0 + k/N\varphi_N$ for $k=1,\ldots,N-1$.
		\item Solve the minimization problem
		\begin{equation*}
			\min_{\varphi_1,\ldots,\varphi_{N-1}\in Q_n} A_d(\varphi_0,\ldots,\varphi_N)
		\end{equation*}
		with a numerical non-linear numerical minimization algorithm.
	\end{enumerate}
\end{algorithm}

\begin{remark}
	In practical computations we use $\C^n$ instead of $Q_n$.
	Thus, as a last step one have to check that the solution obtained fulfils that $\varphi_k\in Q_n$, i.e., $\varphi_k'(z)\neq 0$ for~$z\in\D$.
	For short enough geodesics, this is guaranteed by the fact that $\varphi_0'(z) = 1$.
\end{remark}

\begin{remark}
	Notice that we solve the problem as a two point boundary value problem. 
	Thus, we assume that the final state $\varphi_N$ is known. 
	In future work we will consider the more general optimal control problem, where the final configuration is determined by minimimizing a functional, such as sum-of-squares of the pixel difference between the destination and target image.
\end{remark}

\subsection{Efficient Evaluation of the Discrete Action} 
\label{sub:efficient_computations_using_fft}

Evaluation of the discrete action functional~\eqref{eq:discrete_action}
requires computation of the Taylor coefficients of
the product~$\varphi_{k+1/2}'(\varphi_{k+1}-\varphi_k)$.
The ``brute force'' algorithm for doing this requires
$\mathcal{O}(n^2)$ operations. However, it is well known
that FFT techniques can be used to accelerate such computations.
Using this, we now give an $\mathcal{O}(N n \log n)$ algorithm for evaluation
of the discrete action.

\begin{algorithm}
	Given $\varphi_0,\ldots, \varphi_N$, an efficient algorithm
	for computing the discrete action $A_d(\varphi_0,\ldots,\varphi_N)$
	is given by:
	\begin{enumerate}
		\item For each $k=0,\ldots,N$, compute the Taylor coefficients
		of $\varphi_k'$. This requires $\mathcal{O}(Nn)$ operations.
		\item Compute
		\begin{equation*}
			A_1 = \sum_{k=0}^{N-1} \frac{\alpha}{2 h}\inner{
				\varphi_{k+1}'-\varphi_{k}',
				\varphi_{k+1}'-\varphi_{k}'
			}_{L^2(\D)} .
		\end{equation*}
		This requires $\mathcal{O}(N n)$ operations.
		\item Set $\vect{a}_k \in \C^{2n}$ to contain the Taylor
		coefficients of $\varphi_{k+1/2}'$ as its first $n-1$
		elements, and then zero padded.
		This requires $\mathcal{O}(N n)$ operations.
		\item Set $\vect{b}_k \in \C^{2n}$ to contain the Taylor
		coefficients of $(\varphi_{k+1}-\varphi_k)/h$ as its first $n$
		elements, and then zero padded.
		This requires $\mathcal{O}(N n)$ operations.
		\item Compute
		\begin{equation*}
			\hat{\vect{a}}_k = \mathrm{FFT}(\vect{a}_k),
			\quad
			\hat{\vect{b}}_k = \mathrm{FFT}(\vect{b}_k).
		\end{equation*}
		This requires $\mathcal{O}(N n \log n)$ operations.
		\item Compute component-wise multiplication
		\begin{equation*}
			\hat{\vect{c}}_k = \hat{\vect{a}}_k \hat{\vect{b}}_k .
		\end{equation*}
		This requires $\mathcal{O}(N n)$ operations.
		\item Compute
		\begin{equation*}
			\vect{c}_k = \mathrm{IFFT}(\hat{\vect{c}}_k) .
		\end{equation*}
		This requires $\mathcal{O}(N n \log n)$ operations.
		\item Let $\phi_k$ be the polynomial with Taylor coefficients
		given by $\vect{c}_k$. Then compute
		\begin{equation*}
			A_2 = \frac{h}{2}\sum_{k=0}^{N-1} \inner{\phi_k,\phi_k}_{L^2(\D)} .
		\end{equation*}
		This requires $\mathcal{O}(N n)$ operations.
	\end{enumerate}
	Finally, the discrete action is now given by
	$
		A_d(\varphi_0,\ldots,\varphi_N) = A_1 + A_2 
	$.
\end{algorithm}



\section{Experimental Results}
\label{sec:results}

In this section we use the numerical method developed in the previous
section to confirm that the discrete Lagrangian method is able to calculate
geodesics (solutions to~\eqref{eq:Euler_Lagrange_strong_disk}) with moderately large deformations. 
The purpose of these experiments is to demonstrate our numerical method: we leave further investigations of conformal
image registration for another paper. In all the examples we compute geodesics starting from the identity
$\varphi_0(z) = z$ and ending at some polynomial $\varphi_1$
(of relatively low order). We consider both simple linear
and heavily non-linear warpings. The simulations are carried
out with two different values of the parameter~$\alpha$ to illustrate how the 
geodesics depend on the metric.
The data is given in \autoref{tab:data}.

\autoref{fig:example12} shows the geodesics corresponding to scaling and rotation.
Notice that the geodesics stay in the submanifold $\mathrm{Lin}(\D,\R^2)$, as predicted
by \autoref{pro:tg_submanifold1}. Also notice the difference between
large and small $\alpha$. For small $\alpha$, the scaling coefficient
behaves (almost) like $\frac{\ud^2}{\ud t^2}c_1^2 = 0$, which is the solution
of equation~\eqref{eq:tg_strong_eq} with $\alpha=0$, while the scaling coefficient behaves (almost) like $\frac{\ud^2}{\ud t^2}c_1 = 0$,
which is the asymptotic solution of equation~\eqref{eq:tg_strong_eq} as $\alpha\to\infty$.

\autoref{fig:example34} shows the geodesics corresponding
to various non-linear transformations. Although the differences in the geodesic paths
for different values of $\alpha$ are small, we notice that for higher values of $\alpha$
the geodesic are ``more regular'' at the cost of occupying more volume.
This is especially clear in Example~3, where, halfway thorough the geodesic, 
the ``bump'' on the right of the shape behaves differently for the two values of~$\alpha$.

We anticipate that the method will allow us to study the metric geometry of the conformal embeddings as was done by Michor and Mumford~\cite{MiMu2006} for metrics on planar curves and to determine, for example, which conformal embeddings are markedly closer in one metric than another, and how the geodesic paths differ between different metrics and between different groups, for example, by passing to a smaller group (e.g. the M\"obius group) or to a larger one (e.g. the full diffeomorphism pseudogroup for embeddings). 
\begin{table}
	\begin{center}
		\begin{tabular}{p{0.23\textwidth}p{0.37\textwidth}p{0.23\textwidth}p{0.13\textwidth}}
			\toprule 
			\bf Annotation & \bf Coefficients & \bf Choice of $\alpha$  & \bf Type \\
			\midrule
				Example 1 (a,b) 
			& 
				$c_0 = 0$ \par
				$c_1 = 0.5$
			& 
				$\alpha=0.1$ in (a) 
				\par 
				$\alpha=100$ in (b)
			&
				Scaling
			\\
			\midrule
				Example 2 (a,b) 
			& 
				$c_0 = 0$ \par
				$c_1 = \exp(0.4 \pi \i)$
			& 
				$\alpha=0.1$ in (a) 
				\par 
				$\alpha=100$ in (b)
			&
				Rotation
			\\
			\midrule
				Example 4 (a,b) 
			& 
				$c_0 = 0.0185475$
				\par
				$c_1 = 0.8034225$
				\par
				$c_2 = -0.13933275$
				\par
				$c_3 = -0.23849625$
				\par
				$c_4 = -0.18597975$
				\par
				$c_5 = -0.0125472$
				\par
				$c_6 = 0.18020775$
				\par
				$c_7 = 0.27937125$
			& 
				$\alpha=0.1$ in (a) 
				\par 
				$\alpha=10$ in (b)
			&
				Non-linear
			\\
			\midrule
				Example 5 (a,b) 
			& 
			$c_0 = 0.00674 + 0.053125\i$
			\par
			$c_1 = 0.77654 + 0.103125\i$
			\par
			$c_2 = 0.109424 + 0.103125\i$
			\par
			$c_3 = -0.052777 + 0.103125\i$
			\par
			$c_4 = -0.115049 + 0.103125\i$
			\par
			$c_5 = -0.0409141 + 0.103125\i$
			\par
			$c_6 = 0.126201 + 0.103125\i$
			\par
			$c_7 = 0.288402 + 0.103125\i$
			& 
				$\alpha=0.1$ in (a) 
				\par 
				$\alpha=10$ in (b)
			&
				Non-linear
			\\
			\bottomrule
		\end{tabular}
	\end{center}
	\caption{Data used in the various examples. 
	The polynomial for the final conformal mapping 
	is of the form $\varphi_1(z) = \sum_{k=0}^{n-1} c_k z^k$
	with $n=16$. Coefficients not listed are zero.
	In all examples,
	we used $N=20$ discretisation points in time.
	}
	\label{tab:data}
\end{table}

\begin{figure}
	\centering
		\textbf{Example 1} \\ (a) \\
		\includegraphics[width=1.0\textwidth]{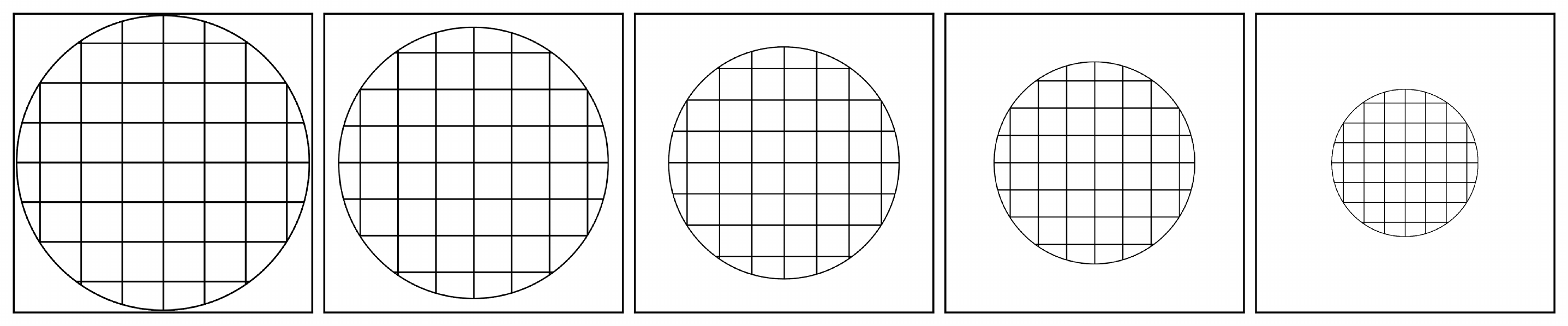}
		\\ (b) \\
		\includegraphics[width=1.0\textwidth]{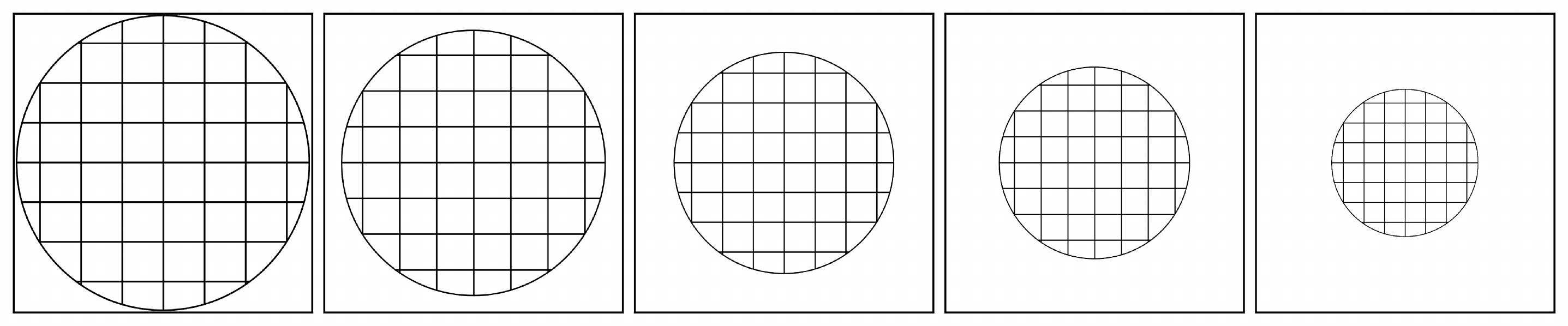}
		\\[3ex]
		\textbf{Example 2} \\ (a) \\
		\includegraphics[width=1.0\textwidth]{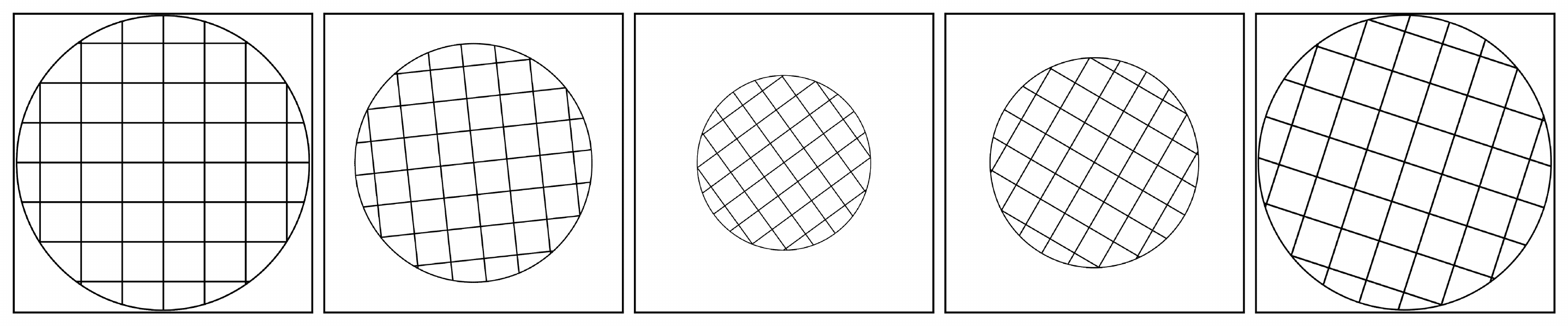}
		\\ (b) \\
		\includegraphics[width=1.0\textwidth]{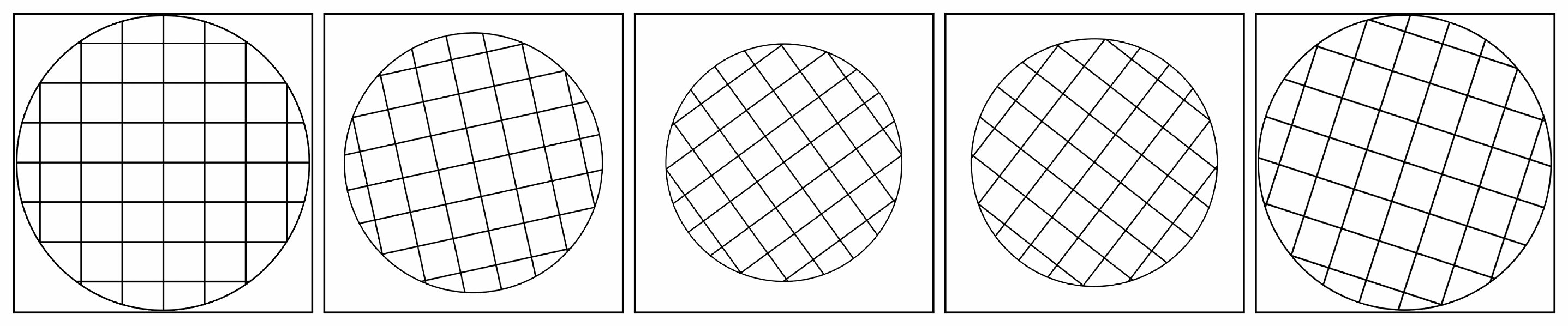}
	\caption{In Example 1, geodesics in the $H^1_\alpha$ metric connecting the identity $z\mapsto z$ to $z\mapsto 0.5z$ are calculated using the discrete Lagrangian method. In (a), $\alpha= 0.1$ and in (b), $\alpha =100$. Both geodesics coincide with the analytic solution to equation~\eqref{eq:tg_strong_eq}. In Example 2, the geodesic connecting $z\mapsto z$ to $z\mapsto \e^{0.4\pi \i}z$ is calculated, again matching the analytic solution perfectly, illustrating the effect of $\alpha$.
	\label{fig:example12}}
\end{figure}

\begin{figure}
	\centering
		\textbf{Example 3} \\ (a) \\
		\includegraphics[width=1.0\textwidth]{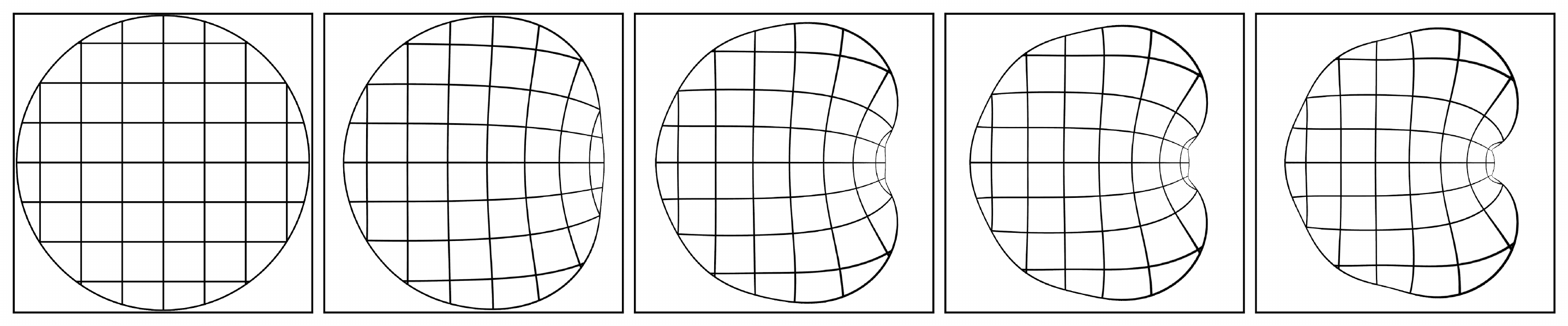}
		\\ (b) \\
		\includegraphics[width=1.0\textwidth]{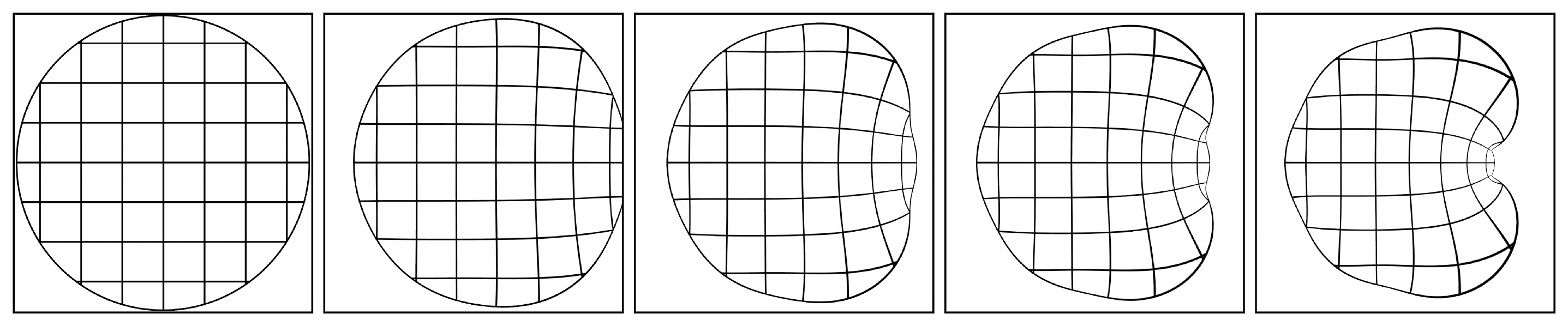}
		\\[3ex]
		\textbf{Example 4} \\ (a) \\
		\includegraphics[width=1.0\textwidth]{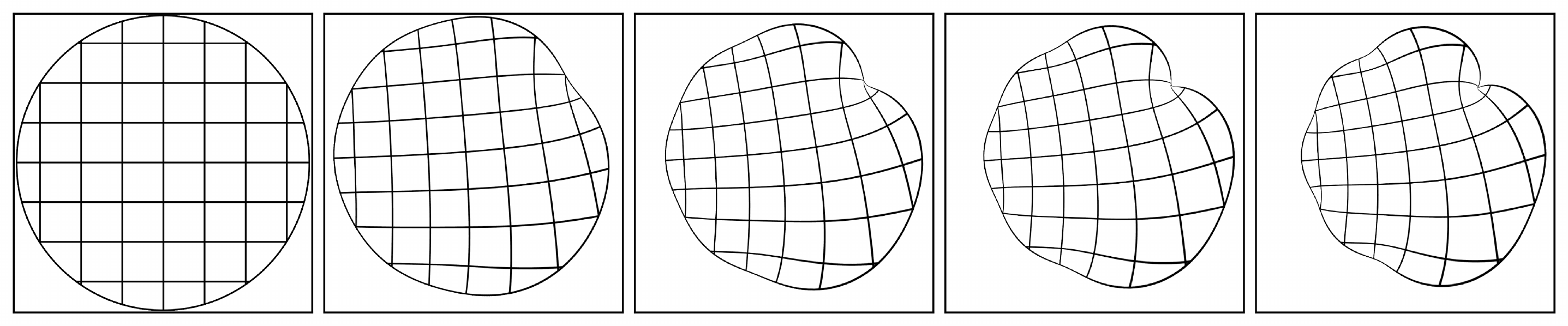}
		\\ (b) \\
		\includegraphics[width=1.0\textwidth]{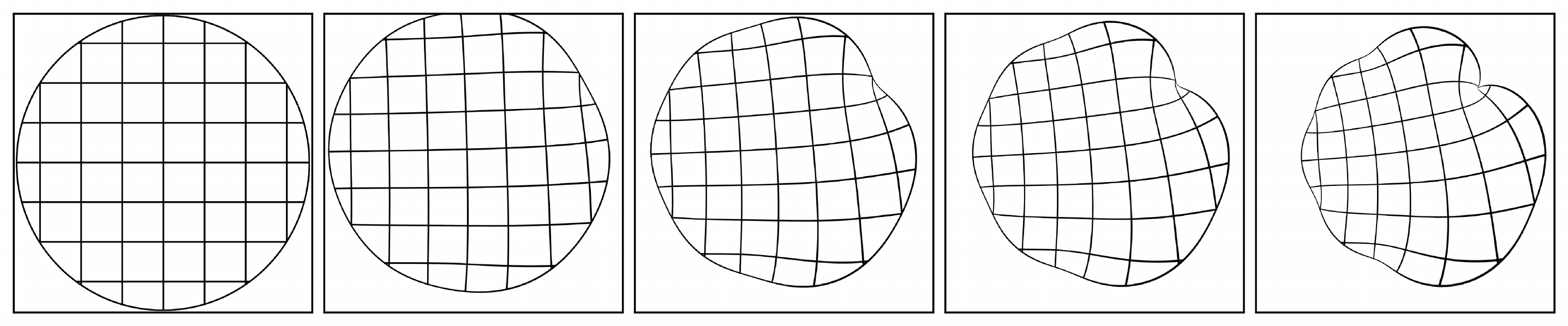}
	\caption{Examples 3 and 4 illustrate two geodesics in the $H^1_\alpha$ metric on conformal embeddings calculated using the discrete Lagrangian method. In (a), $\alpha = 0.1$ and in (b), $\alpha = 100$. 
	In both examples, the target diffeomorphism $\varphi_1$ has been chosen to be highly non-linear (see \autoref{tab:data} for the exact data used). 
	Little difference is visible to the eye between the two values of $\alpha$: in Example 3, a small bump on the right side of the boundary behaves differently.
	\label{fig:example34}}
\end{figure}

\section{Conclusions}

Motivated by the preference of Thompson \cite{Th1942} for `simple' warps in his examples of how images of one species can be deformed into those of a related species, we have derived the geodesic equation for planar conformal diffeomorphisms using the $H^1_{\alpha}$ metric. We have chosen conformal warps as they were used by Thompson, and are very simple. Of course, the animals that Thompson was interested in are actually three-dimensional, and for any number of dimensions bigger than 2 the set of conformal warps is rather restricted. However, our intention is to start with conformal warps and to continue to build progressively more complex sets of deformations, rather than working always in the full diffeomorphism group as is conventional.

The conformal warps admit the rigid transformations of rotation and translation as special cases, and we have shown that these linear conformal transformations are totally geodesic in the conformal warps with respect to the $H^1_{\alpha}$ metric that we have considered. 

We have also provided a numerical discretization of the  geodesic equation, and used it to demonstrate the effects of the parameters of the conformal warps. In future work we will use the algorithm that we have developed here to perform image registration based on conformal warps using the LDDMM framework and apply it to images such as real examples of those drawn by Thompson. 

\section*{Acknowledgements}

This work was funded by the Royal Society of New Zealand Marsden Fund and the Massey University Postdoctoral Fellowship Fund.
The authors would like to thank the reviewers for helpful comments and suggestions.


{
\bibliographystyle{splncs}

\bibliography{../Papers/References}
}

\end{document}